\documentclass[11pt]{amsart}
\usepackage{amsfonts,amssymb,amsmath,amsthm}
\usepackage{url}
\usepackage{enumerate}
\usepackage[all]{xy}

\usepackage{amsmath,amsfonts,amsthm,url,color,amssymb}
\usepackage{graphicx}
\usepackage{algpseudocode, algorithm}
\usepackage{hyperref}
\usepackage{todonotes}

\newcommand{\GL}{\mathrm{GL}}

\newcommand{\fqn}{\mathbb{F}_{q^n}}
\newcommand{\F}{\mathbb{F}}
\newcommand{\fq}{\mathbb{F}_q}

\newcommand{\order}[1]{\left|#1\right|}

\newtheorem{theorem}{Theorem}[section]
\newtheorem{lemma}[theorem]{Lemma}

\newtheorem{corollary}[theorem]{Corollary}
\newtheorem{proposition}[theorem]{Proposition}
\theoremstyle{definition}
\newtheorem{definition}[theorem]{Definition}
\newtheorem{example}[theorem]{Example}

\theoremstyle{remark}
\usepackage{amssymb}
\newtheorem{remark}[theorem]{Remark}
\usepackage{enumerate}
\author[L. Reis]{Lucas Reis}
\address{Departamento de Matem\'{a}tica, Universidade Federal de Minas Gerais, Belo Horizonte MG, Brazil}
\email{lucasreismat@gmail.com}
\author[Q. Wang]{Qiang Wang}
\address{School of Mathematics and Statistics, Carleton University, Ottawa, Ontario, K1S 5B6, Canada}
\thanks{The research of the author is partially supported by NSERC of Canada (RGPIN- 2017-06410)}
\email{wang@math.carleton.ca}

\begin{document}

\title[Additive index of polynomials]{The additive index of polynomials over finite fields} 

\begin{abstract}
In this paper we introduce the additive analogue of the index of a polynomial over finite fields. We study several problems in the theory of polynomials over finite fields  in terms of their additive indices,  such as value set sizes, bounds on multiplicative character sums,  and characterizations of  permutation polynomials.
\end{abstract}

\maketitle

\section{Introduction}

Let $q=p^n$ be a power of a prime number $p$ and let $\F_q$ be the finite field of $q$ elements.  
It is well known that every polynomial $P$ over $\mathbb{F}_q$ such that $P(0) =b$
has the form $ax^rf(x^s)+b$ for some positive integers
 $r, s$ such that $s\mid (q-1)$. There are different ways to choose $r, s$ in the form
$ax^rf(x^s)+b$.   However,   motivated by the study of cyclotomic mappings in \cite{NW:05} and
\cite{Wang:07}, a unique way in terms of the index of a polynomial was introduced in \cite{AGW:09}. 
To be more precise, given  
$$P(x)=a(x^d+a_{d-i_1} x^{d-i_1}+\cdots+a_{d-i_k} x^{d-i_k})+b,$$
where $a,~a_{d-i_j}\neq 0$, $i_0 = 0  <  i_1 < \cdots < i_k < d$,  $j=1, \dots, k$.  The case that $k=0$ is trivial. Thus we shall assume that
$k\geq 1$. Write $d-i_k=r$, the vanishing order of $x$ at $0$ (i.e., the lowest degree of $x$ in $P(x)-b$ is $r$).
Then $P(x)=a x^r f(x^{(q-1)/\ell}) +b,$ where $f(x)=
x^{e_0}+a_{d-i_1} x^{e_1}+\cdots+ a_{d-i_{k-1}}x^{e_{k-1}} + a_{r}
$,  $s= \gcd(d-r,d-r-i_1,\dots, d-r-i_{k-1}, q-1)$, $d-r = e_0 s$,  $d-r-i_j = e_j s$, $1\leq j\leq k-1$, and 
  $\ell :=\frac{q-1}{s}$. 
Hence in this case $\gcd(e_0, e_1, \dots, e_{k-1}, \ell)=1$. The integer $\ell=\frac{q-1}{s}$ is called the
  {\it index} of $P(x)$.  One can see that the greatest common divisor condition in the defintion of $s$ makes the index $\ell$ minimal among those possible choices.
    
This  notion of the index of a polynomial over a finite field  was  first  introduced   to study the distribution of permutation polynomials over finite fields \cite{AGW:09}. It turns out that this parameter is also useful in studying  value set size bounds  \cite{WanWang},   character sum bounds \cite{MWW14},  Carltiz rank \cite{IsikWinterhof},  among others; see a recent survey \cite{Wang:19} and the references therein for more details. 

In this paper we introduce the additive analogue of this notion so that we can also write any polynomial over finite field uniquely in terms of its additive index. Namely,  $P(x)=f(\mathcal L(x))+M(x)$  where   $\mathcal L(x), M(x)$ are $p$-linearized polynomials over $\F_{q}$, $\deg(M)<\deg(\mathcal L)$ and $\mathcal L(x)$ splits completely over $\F_{q}$. 
 In fact, we can provide a simple method to compute the additive index: for more details, see Theorem~\ref{thm:index} and Corollary 3.4. As consequences,  we  study the value set size problem and multiplicative character sum problem for polynomials  in terms of their additive indices. In the former case,  we reduce the problem of finding value set size of a polynomial to counting the number of images of distinct coset representatives of certain subspaces determined by its addtive index.  This reduction is nontrivial as long as the additive index of the polynomial is nontrivial. In particular,  if  a polynomial of additive index $n-k$  is  not a PP then its value set size is upper bounded by $p^n-p^k$.  In the latter case, we reduce the multiplicative character sum over the finite field to the sum over certain affine space and thus obtain some nontrivial upper bounds which also improves the well known Weil's bound in many cases.  We also provide some structural and construction results on permutation polynomials and their compositional invereses based on decompositions that describe the additive index.  Special permutations with certain cycle structures are also studied.  

Here goes the structure of the paper. In Section 2 we provide some preliminary results on linearized polynomials that might be of independent interest. In Section 3 we introduce the additive index of a polynomial and provide an interesting method of computing it. In Section 4 we provide   applications of the additive index to the  value set size problem and then character sum bound problem respectively. Finally the results of permutation polynomials and their compositional inverses are addressed in Section 5.

\section{Preparation}
This section provides some machinery that is further used. Let $Q$ be a power of a prime $p$. By a $Q$-linearized polynomial we mean a polynomial of the form $\sum_{i=0}^ma_ix^{Q^i}$, where $a_i\in \overline{\F}_p$. If $f(x)=L(x)+c$ with $L$ a $Q$-linearized polynomial, then $f$ is $Q$-affine. 
From the well-known identity $(a+b)^Q=a^Q+b^Q$, we see that $Q$-linearized polynomials with coefficients in $\F_{Q^t}$ induce $\F_Q$-linear maps over every finite extension of $\F_{Q^t}$. We have the following technical results.

\begin{lemma}\label{lem:b1}
If $L\in \overline{\F}_{p}[x]$ is a polynomial such that $L(x+a)=L(x)+L(a)$ for every $a\in\overline {\F}_p$ and $L(b x)=b L(x)$ for every $b\in \F_Q$, then $L$ is $Q$-linearized.
\end{lemma}

\begin{proof}
If $L$ is the zero polynomial, the result trivially holds. Suppose that $L$ is not the zero polynomial and let $d$ be its degree. Let $t$ be sufficiently large in a way that $Q^t>d$ and $L\in \F_{Q^t}[x]$. Therefore, $L$ induces an $\F_Q$-linear map on $\F_{Q^t}$. However, the number of $\F_{Q}$-linear maps from $\F_{Q^t}$ to itself equals $Q^{t^2}$, the number of $t\times t$ matrices with entries in $\F_{Q}$. This number coincides with the number of $Q$-linearized polynomials of degree at most $Q^{t-1}$, defined over $\F_{Q^t}$. Since $L\in \F_{Q^t}[x]$ and $d<Q^t$, it follows that $L$ is a $Q$-linearized polynomial. 
\end{proof}

\begin{lemma}\label{lem:b2}
Given two $Q$-linearized polynomials $L, M$ such that $L$ is separable, we have that $L$ divides $M$ if and only if there exists another $Q$-linearized polynomial $N$ such that $M(x)=N(L(x))$. 
\end{lemma}

\begin{proof}
It suffices to prove the ``only if'' part since $N(L(x))$ is divisible by $L(x)$ whenever $N(0)=0$. Let $V$ be the set of roots of $L$, hence $L(x)=\prod_{v\in V}(x-v)$. If $L$ divides $M$, it follows that $M(v)=0$  for every $v\in V$. Since $M$ is $Q$-linearized, we obtain the following identity in $\overline{\F}_p[x]$ 
$$M(x+v)=M(x)+M(v)=M(x).$$ 
In other words, the polynomial $M$ is invariant by the translations $x\mapsto x+v, v\in V$. From Theorem 2.5 in~\cite{LR17}, it follows that $M(x)=N(L(x))$ for some polynomial $N$. So it remains to show that $N$ is $Q$-linearized. For this, pick $x_1, x_2 \in \overline{\F}_p$ arbitrary and let $a_1, a_2\in \overline{\F}_p$ be such that $L(a_i)=x_i$, hence 
$$0=M(a_1+a_2)-M(a_1)-M(a_2)=N(x_1+x_2)-N(x_1)-N(x_2).$$
Therefore, $N(x+a)=N(x)+N(a)$ for every $a\in \overline{\F}_p$. A similar argument entails that $N(b x) =b N(x)$ for every $b\in \F_Q$. From Lemma~\ref{lem:b1}, it follows that $N$ is $Q$-linearized.

\end{proof}

We obtain the following result.

\begin{theorem}\label{thm:aux}
Let $L\in \F_{Q^n}[x]$ be a $Q$-linearized polynomial dividing $x^{Q^n}-x$ and let $V\subseteq \F_{Q^n}$ be the set of its roots. For any polynomial $f\in \F_{Q^n}[x]$, the following are equivalent:

\begin{enumerate}[(i)]
\item $f(x+v)=f(x)+f(v)$ and $f(b v)=b f(v)$ for every $v\in V$ and $b\in \F_Q$;
\item there exists a polynomial $T\in \F_{Q^n}[x]$ with $T(0)=0$ and a $Q$-linearized polynomial $M\in \F_{Q^n}[x]$ of degree at most $\deg(L)-1$ such that $f(x)=T(L(x))+M(x)$.
\end{enumerate}
\end{theorem}

\begin{proof}
The direction (ii)$\Rightarrow$(i) follows directly by calculations. For the direction (i)$\Rightarrow$(ii), suppose that $f(x+v)=f(x)+f(v)$ and $f(b v)=b f(v)$ for every $v\in V$ and $b\in \F_Q$.  Using Euclidean algorithm, we can easily prove that $f$ can be written uniquely as $f(x)=\sum_{i=0}^eQ_i(x)\cdot L(x)^i$, where $\deg(Q_i)<\deg(L)$ (the expansion of $f$ in basis $L$). From hypothesis, for $v\in V$, we have that
\begin{eqnarray*}
0&=&f(x+v)-f(x)-f(v)\\
&=&\sum_{i=0}^eQ_i(x+v)L(x+v)^i-\sum_{i=0}^eQ_i(x)L(x)^i-\sum_{i=0}^eQ_i(v)L(v)^i.
\end{eqnarray*}
Since $L$ is $Q$-linearized and $L(v)=0$, the latter implies that
$$0=\sum_{i=1}^e(Q_i(x+v)-Q_i(x))L(x)^i+(Q_0(x+v)-Q_0(x)-Q_0(v)).$$
Since the representation in basis $L$ is unique, we necessarily have that $Q_i(x+v)=Q_i(x)$ for any $1\le i\le e$ and 
$$Q_0(x+v)-Q_0(x)-Q_0(v)=0.$$
We claim that $Q_i(x)$ is a constant polynomial if $0<i\le e$. In fact, from  the above, we obtain that $Q_i(0)=Q_i(v)$ for every $v\in V$ and every $0<i\le e$. Hence $Q_i(x)-Q_i(0)$ vanishes at each element of $V$, a set of cardinality $\deg(L)$. But $\deg(Q_i)<\deg(L)$, and so $Q_i(x)$ equals the constant polynomial $Q_i(0)$. In particular, we have that 
$$f(x)=T(L(x))+Q_0(x).$$
with $T(x)=\sum_{i=1}^eQ_i(0)x^i$. Since the polynomial $M_t(x):=Q_0(t+x)-Q_0(t)-Q_0(x)\in \F_{Q^n}(t)[x]$ vanishes at the elements of $V$ and has degree at most $\deg(L)-1<|V|$, it follows that $M_t$ is the zero polynomial. Therefore, $Q_0(x+a)=Q_0(x)+Q_0(a)$ for every $a\in \overline{\F}_Q$. Since  $f(b v)=b f(v)$ for every $v\in V$ and every $b\in \F_Q$, we obtain that $Q_0(bv)=bQ_0(v)$. By a similar argument, we conclude that $Q_0(bx)=bQ_0(x)$ for every $b\in \F_Q$. From Lemma~\ref{lem:b1}, $Q_0$ is a $Q$-linearized polynomial. 
\end{proof}

\begin{remark}\label{rem:aux}
If $Q=p$, the condition $f(bv)=bf(v), b\in \F_Q$ in Theorem~\ref{thm:aux} can be removed, since it is already implied by the condition $f(x+v)=f(x)+f(v)$.
\end{remark}

\section{The additive index of a polynomial}

 Let $p$ be a prime number and $q=p^n$. In this section we explore the decomposition of polynomials in $\F_q[x]$ as $f(L(x)) + M(x)$, where $L(x), M(x)$ are $p$-linearized polynomials. First of all we introduce the following definition. 

\begin{definition}
Let $L\in \fq[x]$ be a {\em subspace polynomial} over $\F_q$, i.e., a monic $p$-linearized polynomial dividing $x^{q}-x$. A polynomial $P\in \fq[x]$ is {\em $L$-decomposable} if it can be written as $P(x)=f(L(x))+M(x)$ in a way that $f, M\in \fq[x]$ and $M$ is a $p$-linearized polynomial with $\deg(M)<\deg(L)$. 
\end{definition}

We observe that every polynomial $P\in \F_q[x]$ is $L$-decomposable for $L(x)=x$; we take $P=f$ and $M$ is the zero polynomial. This is the trivial decomposition of $P$. A polynomial may be $L$-decomposable for various $L$. In fact, if $P$ is $p$-affine, then $P$ is $L$-decomposable for every monic $p$-linearized polynomial $L$ that divides $x^q-x$: this is a direct consequence of Theorem~\ref{thm:aux} since, in this case, $P(x)-P(0)$ is a $p$-linearized polynomial. The following theorem entails that for any polynomial $P$, we have uniqueness if we consider the $L$'s of maximal degree.

\begin{theorem}\label{thm:index}
For $P\in \F_q[x]$, set $P_0(x)=P(x)-P(0)$ and let $V=V(P, q)$ be the set of all elements $y\in \F_q$ yielding the following identity in $\F_q[x]$:
$$P_0(x+y)-P_0(x)-P_0(y)=0.$$
If $\mathcal L=\mathcal L (P, q)\in \F_q[x]$ is defined by $\mathcal L(x)=\prod_{v\in V}(x-v)$, then the following hold:

\begin{enumerate}[(i)]
\item $\mathcal L$ is a {\em subspace polynomial} over $\F_q$, that is, a monic $p$-linearized polynomial that divides $x^q-x$;
\item For any monic $p$-linearized polynomial $L\in \F_q[x]$ dividing $x^{q}-x$, $P$ is $L$-decomposable if and only if $L$ divides $\mathcal L$. 
\end{enumerate}
\end{theorem}

\begin{proof}
We prove items (i) and (ii) separately.

\begin{enumerate}[(i)]
\item It is a routine exercise to prove that the vanishing polynomial of an $\F_p$-vector space contained in $\F_q$ is a $p$-linearized polynomial. So it suffices to prove that $V$ is an $\F_p$-vector space. Clearly $0\in V$, so we only need to prove that $V$ is closed under addition. If $u, v\in V$, then
\begin{equation}\label{eq:sum}P_0(x+u+v)=P_0(x+u)+P_0(v)=P_0(x)+P_0(u)+P_0(v).\end{equation}
Taking $x=0$ in Eq.~\eqref{eq:sum}, we obtain that $P_0(u+v)=P_0(u)+P_0(v)$. Hence Eq~\ref{eq:sum} implies that $P_0(x+u+v)=P_0(x)+P_0(u+v)$ and then $u+v\in V$.
\item Since $\mathcal L$ is the vanishing polynomial of $V$, this item follows from Theorem~\ref{thm:aux} and Remark~\ref{rem:aux}.
\end{enumerate}
\end{proof}

Theorem~\ref{thm:index} motivates us to introduce the following definition.

\begin{definition}
Write $q=p^n$. For a polynomial $P\in \F_q[x]$, let $\mathcal L\in \F_q[x]$ be the $p$-linearized polynomial as in Theorem~\ref{thm:index}.  In other words, $P$ is $\mathcal L$-decomposable where $\mathcal L$ is of the maximal degree. If $\deg({\mathcal L})=p^{n-k}$, then we say that  the polynomial $P$ has {\em additive index} $k$ over $\F_q$.
\end{definition}

From previous observation, any $p$-affine polynomial $P(x)\in \F_q[x]$ is $L$-decomposable for $L(x)=x^q-x$. In particular, such polynomials have additive index $0$: the converse is also true under the condition $\deg(P)<q$. Theorem ~\ref{thm:index} also provides a simple way of finding the additive index of an arbitrary polynomial over $\F_q$ via the GCD of certain polynomials.

\begin{corollary}
For a polynomial $P\in \F_q[x]$ of degree $d$, let $P_0(x)=P(x)-P(0)$ and consider the bivariate polynomial
$$P_0(x+y)-P_0(x)-P_0(y)=\sum_{i=1}^{d-1}F_i(y)x^i.$$

If $\mathcal L$ is as in Theorem~\ref{thm:index}, then 
$$\mathcal L(x)=\gcd(F_1(x), \ldots, F_{d-1}(x), x^q-x).$$
\end{corollary}
\begin{proof}
If $V=V(P, q)$ is as in Theorem~\ref{thm:index}, we observe that $y\in V$ if and only if $y\in \F_q$ and $F_i(y)=0$ for every $1\le i\le d-1$. Equivalently, $y$ is a root of the polynomial $\gcd(F_1(x), \ldots, F_{d-1}(x), x^q-x)$, which is separable since $x^q-x$ is separable. The result follows from the definition of $\mathcal L$.
\end{proof}
From the previous corollary, after computing the $p$-linearized polynomial $\mathcal L$ associated to $P\in \F_q[x]$, we just need to write $P$  (more precisely,  $P(x)-P(0)$) in basis $\mathcal L$.  Theorem~\ref{thm:aux} guarantees that we obtain an expansion 
$$P(x)=f(\mathcal L(x))+M(x),$$
where  $\mathcal L(x), M(x)$ are $p$-linearized polynomials over $\F_{q}$, $\deg(M)<\deg(\mathcal L)$ and $\mathcal L(x)$ splits completely over $\F_{q}$.

%


\section{On $M$-affine mappings}
Here $L$ denotes a monic $p$-linearized polynomial dividing $x^q-x$. We observe that the set $U_L:=\{z\in\overline{\F}_q\,|\, L(z)=0\}\subseteq \fq$ is an $\F_p$-vector space of codimension $k$, where $p^{n-k}=\deg(L)$. In particular, $(U_L, +)$ is a subgroup of the abelian group $(\F_q, +)$.

\begin{definition}
An $L$-coset in $\F_q$ is any set of the form $a+U_L$ with $a\in \F_q$. 
\end{definition}

If $P\in \F_q[x]$ writes as $P(x)=f(L(x))+M(x)\in \F_q[x]$ with $M$ a $p$-linearized polynomial, we observe that the map $c\mapsto P(c)$ takes a simpler form on each $L$-coset. In fact, for $a\in \F_q$ and $u\in U_L$, we have that
\begin{equation}\label{eq:coset}P(a+u)=f(L(a+u))+M(a+u)=f(L(a))+M(a+u)=P(a)+M(u).\end{equation}

\begin{definition}
Let $U_0\subseteq \F_{q}$ be an $\F_p$-vector space of dimension $n-k$. Then $\F_{q}$ can be partitioned into $U_0, \ldots, U_{p^{k}-1}$, where each $U_i$ is of the form $\zeta_i + U_0$ with $\zeta_i \in \F_{q}$.  
For a $p$-linearized polynomial $M\in \F_q[x]$ and a sequence $\{a_i\}_{ 0\leq i \leq p^{k}-1}$ in $\F_q$,  we can define an {\em $M$-affine mapping  $P$ of index $k$}  with the subspace $U_0$ by 

\begin{equation}\label{CycloMappingDef}
P(x) =
\left\{
\begin{array}{ll}
M(x) + a_0,   &   \mbox{if} ~x \in U_0, \\
\vdots & \vdots \\
M(x) + a_i,   &   \mbox{if} ~x \in U_i,\\
\vdots & \vdots \\
M(x) + a_{p^{k}-1}  &   \mbox{if} ~x \in U_{p^{k}-1}.\\
\end{array}
\right.
\end{equation}

\end{definition}

\begin{remark}\label{rem:cyc-index} Eq.~\eqref{eq:coset} entails that  if  $P(x)=f(L(x))+M(x)$  is a polynomial with additive index $k$, i.e.,  $\deg(L)=p^{n-k}$, then $P$ induces an $M$-affine mapping of index $k$ and subspace $U_L=\{z\in \overline{\F}_q \mid  L(z)=0\}\subseteq \F_q$.  Conversely, if $P$ is an $M$-affine mapping of index $k$ with subspace $U_0$, then the polynomial representation of $P$ is $L$-decomposable, where $L(x)=\prod_{u\in U_0}(x-u)$ . We note that a mapping could be represented by $M$-affine mappings with various indices (if we take different subspaces), but the smallest index gives the additive index of the corresponding polynomial.
\end{remark}


In the rest of  this section we provide some applications of the additive index of polynomials over finite fields. These applications concern about some classical problems in the theory: the value set sizes of polynomials  and character sums with polynomial arguments. 
Throughout this section, $L$ and $M$ usually denote $p$-linearized polynomials, where $L(x)$ is monic and  divides $x^q-x$.

\subsection{Value set of polynomials through the additive index}

For $q=p^n$ and $P \in \fq[x]$, $V_P=\{P(y)\,|\, y\in \F_q\}$ denotes the value set of $P$ over $\F_q$. We obtain the following result. 


\begin{theorem}\label{prop:value-set}
Let $L$ be a monic $p$-linearized polynomial that divides $x^q-x$, set $p^{n-k}=\deg(L)$ and $U_L=\{z\in\overline{\F}_q\,|\, L(z)=0\}$. If $M$ is a $p$-linearized polynomial and $P(x)=f(L(x))+M(x)\in \F_q[x]$, then the value set size $|V_P|$  of $P$ over $\F_q$ satisfies
$$|V_P|=c\cdot \frac{p^{n-k}}{\deg(\gcd(L, M))},$$
where $c$ is the number of distinct cosets representatives in $\{P(a)\,|\, a\in \F_q/U_L\}$ for the $\F_p$-vector space $M(U_L)$. 
\end{theorem} 

\begin{proof}
If $P(x)=f(L(x))+M(x)\in \F_q[x]$ with $M$ a $p$-linearized polynomial,  then  $P$ is an $M$-affine mapping of additive index $k$ with  the subspace $U_0= U_L$.  In particuar, since $L$ is separable, the Rank-Nullity Theorem entails that $M$ maps $U_L$ onto an $\F_p$-vector space of cardinality $\frac{p^{n-k}}{\deg(\gcd(L, M))}$. Moreover, $P$ maps each coset $a+U_L$ onto $P(a)+M(U_L)$. The proof is complete.
\end{proof}

A polynomial $P \in \F_{q}[x]$ is a permutation polynomial (PP) if $|V_P|=q$.   It is a well known result due to  Wan  \cite{Wan}  that  if $P$ is not a PP then 
$|V_P| <  q - \frac{q-1}{d}$, where $d$ is the degree of the polynomial $P$.  Later on,  Mullen, Wan, and Wang \cite{MWW14} extended this bound  in terms of the  (mulitiplicative) index of the polynomials.  Namely, if $P$ is not a PP then 
$|V_P| <  q - \frac{q-1}{\ell}$, where $\ell$ is  the multiplicative index of $P$.  The following result is the additive analog of the above result.  Obviously,  if $P(x)=f(L(x))+M(x)$  is a PP of $\F_{q}$ with additive index $n-k$, then $\gcd(L(x), M(x)) $ must be a PP;  without loss of generality, we can assume that $\gcd(L(x), M(x) =x$. 

\begin{corollary}
Let $q=p^n$ and $P(x)=f(L(x))+M(x)$  be a polynomial over
$\F_{q}$ with additive index $k$.
If  $\gcd(L(x), M(x)) =x$ and $|V_P| > p^n - p^{n-k} $, then $P(x)$ is a PP of $\F_{q}$. 
\end{corollary}

Our new bound is very effective for many polynomials  with  large degree $d$ or large multiplicative index $\ell$. Indeed,  if $d > \frac{p^n-1}{p^k} $ or $\ell > \frac{p^n-1}{p^k}$, then our new additive index bound improves
both previous bounds.

\subsection{Additive index bounds on character sums}
Here we use the additive index in order to obtain bounds on character sums with polynomial arguments. The following technical result is required.

\begin{lemma}\label{char-aff}
Write $q=p^n$ and let $\mathcal A\subseteq \F_{q}$ be an $\F_p$-affine space of dimension $e$. If $\eta$ is a nontrivial mutiplicative character over $\F_{q}$, then we have that 
$$\left|\sum_{a\in \mathcal A}\eta(a)\right|\le  p^{\min\{e, n/2\}}.$$
\end{lemma}

\begin{proof}
If $e<n/2$ the bound is trivial and, for $e\ge n/2$, the bound follows by Corollary 3.5 in~\cite{R21}.
\end{proof}

We obtain the following result.

\begin{theorem}\label{thm:char}

Let $P(x) \in \fq[x]$ be a polynomial with additive index $k$ and write $P(x)=f(L(x))+M(x)\in \F_{q}[x]$, where $L(x), M(x)$ are $p$-linearized polynomials over $\F_{q}$, $\deg(M)<\deg(L)=p^{n-k}$  and $L(x)$ splits completely over $\F_{q}$. 
Let  $p^e=\frac{p^{n-k}}{\deg(\gcd(L, M))}$ and let $\eta$ be a nontrivial multiplicative character of $\F_{q}$. Then the following holds:
$$\left|\sum_{x\in \F_{q}}\eta(P(x))\right|\le p^{n-e+\min\{e, n/2\}}.$$ 
\end{theorem}
\begin{proof} Let  $U_L=\{z\in \F_q \mid   L(z)=0\}$ such that  $\dim U_L =n-k$ and  $\fq = \cup_{i=0}^{p^{k}-1} (a_i + U_L)$ for some $a_i \in \fq$ with $0\leq i \leq p^{k} -1$.  Let $V = M(U)$ with  $e= \dim (V)$  and let ${\mathcal A}_i = P(a_i) + V$ for $0\leq i \leq p^{k} -1$.  
We obtain that
 
 \begin{eqnarray*}
\left|  \sum_{x\in \fq} \eta(P(x)) \right| &=& \left|     \sum_{x\in \fq} \eta(f(L(x)) + M(x))     \right| \\
&=& \left|  \sum_{i=0}^{p^{k} -1} \sum_{y\in U_L}   \eta(f(L(a_i+y)) + M(a_i+y))     \right| \\
&=& \left|  \sum_{i=0}^{p^{k} -1} \sum_{y\in U_L}   \eta(f(L(a_i)) + M(a_i) +M(y))     \right| \\
&=& \left|  \sum_{i=0}^{p^{k} -1} \sum_{y\in U_L}   \eta(P(a_i) +M(y))     \right| \\
&=&   \left|  \sum_{i=0}^{p^{k} -1} p^{n-k-e}  \sum_{v\in V}   \eta(P(a_i) + v)  \right| \\
&\leq&    \sum_{i=0}^{p^{k} -1} p^{n-k-e}  \cdot \max \left\{  \left|  \sum_{u\in  {\mathcal A}_i}   \eta(u)  \right| \right\} \\
&\leq  & p^{n-e + \min\{e, n/2\} }, 
\end{eqnarray*}
where the last inequality follows from Lemma~\ref{char-aff}. 

\end{proof}
In the notation of Theorem~\ref{thm:char}, we observe that if $\deg(f)=s\ge 1$ and $P$ is not of the form $ag(x)^r$ for some divisor $r$ of $q-1$, then Weil's bound (see Theorem 5.41 of~\cite{LN}) yields
$$\left|\sum_{y\in \F_{q}}\eta(P(y))\right|\le (sp^{n-k}-1)\cdot p^{n/2}.$$
From construction $e\le n-k$, hence $(sp^{n-k}-1)p^{n/2}\ge (p^e-1)p^{n/2}\ge p^{3n/2-e}=p^{n-e+\min\{e, n/2\}}$ if $e>n/2$. Also, in the range $e>n/2$ we have that  $p^{n-e+\min\{e, n/2\}}=p^{3n/2-e}<p^n$. In summary, when $e>n/2$, the bound in Theorem~\ref{thm:char} is a nontrivial bound, which is also sharper than Weil's bound.

\begin{remark}
To the best of our knowledge, Lemma~\ref{char-aff} is the sharpest known bound considering a generic $\F_p$-affine space in $\F_q$ and a generic nontrivial multiplicative character of $\F_q$. Results on special settings are provided in~\cite{chang, R20}, and they can be applied to Theorem~\ref{thm:char} accordingly. We emphasize that any improvement on Lemma~\ref{char-aff} readily improves Theorem~\ref{thm:char}, which is based on the notion of additive index of a polynomial. 
\end{remark}

\section{Permutation Polynomials}


The following useful criterion first appeared  in \cite{AGW} and then it was further developed in \cite{FFA:YuanD11}, \cite{FFA:YuanD14}, \cite{Zheng16},  \cite{LiQuWang18},  among others.

\begin{lemma}[The AGW Criterion]\label{lemma:AGWCriterion}
Let $A$, $S$ and $\bar{S}$ be finite sets with $\order{S} = \order{\bar{S}}$, and let $f:A\rightarrow A$, $\bar{f}:S\rightarrow \bar{S}$, $\lambda:A\rightarrow S$, and $\bar{\lambda}:A\rightarrow \bar{S}$  be maps such that $\bar{\lambda} \circ f=\bar{f} \circ\lambda$ (see the following commutative diagram).
\[
\xymatrix{
A \ar[r]^{f}\ar[d]^{\lambda} & A \ar[d]^{\bar{\lambda}}\\
S \ar[r]^{\bar{f}} & \bar{S} }
\]

If both $\lambda$ and $\bar{\lambda}$ are surjective, then the following statements are equivalent:
\begin{itemize}
\item $f$ is a bijection from $A$ to $A$ (a permutation over $A$);
\item $\bar{f}$ is a bijection from $S$ to $\bar{S}$ and $f$ is injective on $\lambda^{-1}(s)$ for each $s\in S$.
\end{itemize}
\end{lemma}

In  particular, for any polynomial  $g \in  \fq[x]$, any additive polynomials $\varphi, \psi,\bar{\psi}\in\fq[x]$ satisfying $\varphi \circ \psi = \bar{\psi} \circ \varphi$ and $\#\psi(\fq)=\#\bar{\psi}(\fq)$, and any polynomial $h \in \fq[x]$ such that $h(\psi(\F_{q}))\subseteq \F_p^*$,   the permutation polynomials of the form $f(x) := h(\psi(x)) \varphi (x) + g(\psi(x))$ over $\F_{q}$ were characterized. 

Let $L$ be a monic $p$-linearized polynomial that divides $x^q-x$ and $M$ be a $p$-linearized polynomial such 
that $L(x) \mid L(M(x))$. By Lemma~\ref{lem:b2},   $L(x) \mid L(M(x))$ if and only if there exists another $p$-linearized polynomial $N(x)$ such that $L(M(x)) = N(L(x))$.   From the AGW criterion, we can derive the following result.

\begin{theorem}\label{cor:PP2}
Let $L$ be a monic $p$-linearized polynomial that divides $x^q-x$ and $M$ is a $p$-linearized polynomial such 
that $L(x) \mid L(M(x))$. Then  there exists a $p$-linearized polynomial $N\in \F_q[x]$ such that $N(L(x)) = L(M(x))$. In this case, the polynomial $P(x)=f(L(x))+M(x)\in \F_q[x]$ permutes $\F_q$ if and only if the following conditions hold:
\begin{enumerate}[(i)]
\item $\gcd(L(x), M(x))=x$; and 
\item  $L(f(x)) + N(x)$  is a bijection of $L(\F_q)$. 
\end{enumerate}
In particular,  $P(x)=f(L(x))+x\in \F_q[x]$ is a permutation polynomial over $\F_q$ if  and only if $L(f(x)) + x$ is a bijection of $L(\F_q)$. 
\end{theorem}

We remark that if $M$ permutes the roots of $L$, then $L(x) \mid L(M(x))$. A slightly more general result can be obtained using Theorem~\ref{prop:value-set}. 

\begin{corollary}\label{cor:PP}
Let $L$ be a monic $p$-linearized polynomial that divides $x^q-x$, set $p^{n-k}=\deg(L)$ and $U_L=\{z\in\overline{\F}_q\,|\, L(z)=0\}$. If $M$ is a $p$-linearized polynomial and $P(x)=f(L(x))+M(x)\in \F_q[x]$, then $P$ is a permutation polynomial over $\F_q$ if and only if the following conditions hold:
\begin{enumerate}
\item $\gcd(L(x), M(x))=x$ or, equivalently, $M$ restricts to a bijection from $U_L$ to the set $M(U_L)$;
\item the induced map $\varphi: \F_q/U_L\mapsto \F_q/M(U_L)$ given by $\zeta \mapsto P(\zeta)$ is a bijection.
\end{enumerate}
\end{corollary}
\begin{proof}
If $V_P$ denotes the value set of $P$ over $\F_q$, we observe that $P$ is a PP over $\F_q$ if and only if $|V_P|=q$. 
In the notation of Theorem~\ref{prop:value-set}, the latter holds if and only if $c=p^{k}$ and $\deg(\gcd(L(x), M(x)))=1$. The result follows since $0\in \F_q$ is always a common root of $L$ and $M$.
\end{proof}

Motivated by Theorem~\ref{cor:PP2}, we provide the following class of PP's of $\F_q$ where the cycle decomposition can be implicitly computed.

\begin{theorem}\label{cor:conj}
Let $L$ be a monic $p$-linearized polynomial that divides $x^q-x$ and let $f\in \F_q[x]$ be a polynomial such that $L(f(L(y)))=0$ for every $y\in \F_q$. Then $P(x)=f(L(x))+x$ permutes $\F_q$. Moreover, if $t$ denotes the number of distinct roots of $f$ in the set $L(\F_{q})$, then $P(x)$ decomposes into $t\cdot \deg(L)$ cycles of length $1$ and $\frac{q-t\cdot \deg(L)}{p}$ cycles of length $p$.
\end{theorem}

\begin{proof}
From hypothesis, $L(P(z))=L(z)$ for every $z\in \F_q$. In particular, for any $y\in \F_q$, we have that 
$$P(P(y))=f(L(P(y)))+P(y)=f(L(y))+f(L(y))+y=2f(L(y))+y.$$
It follows by induction that, for each positive integer $j\ge 1$, the $j$-th fold composition $P^{(j)}$ satisfies $P^{(j)}(y)=j\cdot f(L(y))+y$ for every $y\in \F_q$. In particular, taking $j=p$, we have that $P^{(p)}(y)=y$ for every $y\in \F_q$. Hence $P$ is a permutation of $\F_q$ and, since $p$ is a prime, all cycles of $P$ are of length $1$ and $p$. The cycles of length $1$ equals the number of $y\in \F_q$ such that $P(y)=y$, i.e., $f(L(y))=0$. If $\alpha\in L(y)$ is a root of $f$, we have $\deg(L)$ solutions $z\in \F_q$ to the equation $L(z)=\alpha$. Therefore, if $f$ has $t$ distinct roots in the set $L(\F_{q})$,  then $P$ has $t\cdot \deg(L)$ cycles of length $1$, from where the result follows.
\end{proof}

\begin{remark}If $L(x)$ is a monic $p$-linearized polynomial that divides $x^q-x$, Lemma~\ref{lem:b2} implies that $x^q-x=\tilde{L}(L(x))$  
for some monic $p$-linearized polynomial $\tilde{L}\in \F_q[x]$; the polynomial $\tilde{L}$ is easily obtained by expanding $x^q-x$ in basis $L(x)$. 
In the proof of Lemma 3.4 in~\cite{R21} it is proved that, in fact, we have that $L(\tilde{L}(x))=x^q-x$, i.e., the polynomials $L$ and $\tilde{L}$ commute. In particular, in the context of Theorem~\ref{cor:conj}, one may take $f$ as any polynomial of the form $\tilde{L}(g(x))$ with $g\in \fq[x]$.
\end{remark}

For instance, if $q=p^n$ with $n\ge 1$ and $L(x)=x^p-x$,  then we have that $\tilde{L}(x)=x^{p^{n-1}}+\cdots+x$ is just the absolute trace polynomial from $\F_q$ over $\F_p$. A similar example arise from the polynomials $L_b(x)=x^p-b^{p-1}x$ with $b\in \F_q^*$, but they can be covered by the former through the conjugation with the permutation $x\mapsto bx$. In~\cite{R18}, the author explores linearized polynomials $L\in \F_q$ such that $L(L(y))=0$ for every $y\in \F_q$, called {\em $2$-nilpotent linearized polynomials} (2-NLP's). In particular, from Example 2.5 of~\cite{R18}, we obtain a non trivial explicit instance of Theorem~\ref{cor:conj}. 

\begin{example}
Let $m$ be a positive integer, $n=2m$ and $q=p^n$. Then for every $f\in \F_{q}[x]$, we have that $P(x)=L(f(L(x)))+x$ permutes $\F_q$, where $L(x)=\alpha\beta x^{p^m}+\alpha x$ with $\alpha, \beta\in \F_q^*$ satisfying $\alpha^{p^m}+\alpha=0$ and $\beta^{p^m+1}=1$.
\end{example}

We observe that if $P$ is a permutation of $\F_q$ that decomposes into cycles of length $1$ and $p$, then the number of cycles of length $1$ must be divisible by $p$. As follows, we prove that there exist permutations of $\F_q$ arising from Theorem~\ref{cor:conj} with any such admissible cycle decomposition.

\begin{corollary}\label{cor:cycle}
Write $q=p^n$ and let $0\le s\le q$ be an integer divisible by $p$. Then there exists a monic $p$-linearized divisor $L$ of $x^q-x$ and $f\in \F_q[x]$ such that $P(x)=f(L(x))+x$ is a permutation of $\F_q$ that decomposes into $s$ cycles of length $1$ and $\frac{q-s}{p}$ cycles of length $p$.
\end{corollary}

\begin{proof}
Write $s=p^j\cdot u$ so that $\gcd(p, u)=1$, hence $1\le j\le n$ and $u\le p^{n-j}$. Let $L(x)$ be the vanishing polynomial of an arbitrary $\F_p$-vector space  $V\subseteq \F_q$ with dimension $j$, i.e., $L(x)=\prod_{v\in V}(x-v)$. Hence $L(x)$ is a monic $p$-linearized polynomial of degree $p^j$ that divides $x^q-x$ and, by Rank-Nullity Theorem, we have that $L(\F_q)$ is a set of cardinality $p^{n-j}$. Since $j\ge 1$, $V\ne \{0\}$ and so, by Lagrange's interpolation, there exists $f\in \F_q[x]$ such that $f$ maps the set $L(\F_q)$ to the set $V$ in a way that exactly $u\le p^{n-j}$ elements of $L(\F_{q})$ are mapped to $0\in V$. With this setting, $L(f(L(y))=0$ for every $y\in \F_q$ and $f$ has exactly $u$ distinct roots in the set $L(\F_{q})$. The result follows from Theorem~\ref{cor:conj}.
\end{proof}

\subsection{Inverse of PP's} The following theorem provides an implicit way of obtaining the inverse of a PP based on its decomposition as $f(L(x))+M(x)$. In particular, we conclude that the additive indices of a PP and its inverse coincide.

\begin{theorem}\label{lem:inversePP}
Let $L$ be a monic $p$-linearized polynomial that divides $x^q-x$, set $p^{n-k}=\deg(L)$ and $U_L=\{z\in\overline{\F}_q\,|\, L(z)=0\}$. Let $\{\zeta_i\}_{0\le i\le p^{k}-1}$ be any complete set of representatives for the quotient $\F_q/U_L$. Suppose that $P(x)=f(L(x))+M(x)\in \F_q[x]$ is a PP over $\F_q$, where $M$ is a $p$-linearized polynomial. Then $P_0(x)=f_0(L_0(x))+M_0(x)\in \F_q[x]$ is the inverse PP of $P$ over $\F_q$, where $f_0, L_0$ and $M_0$ are given as follows
\begin{enumerate}[(i)]
\item $L_0(x)=\prod_{u\in M(U_L)}(x-u)$;
\item $M_0\in \F_q[x]$ is the unique $p$-linearized polynomial of degree at most $p^{n-k}-1$ such that $M_0(M(u))=u$ for every $u\in U_L$;
\item $f_0\in \F_q[x]$ is the unique polynomial of degree at most $p^{k}-1$ such that $f_0(L_0(P(\zeta_i)))+M_0(P(\zeta_i))=\zeta_i$.
\end{enumerate}
In particular, the additive indices of $P$ and $P_0$ coincide.
\end{theorem}

\begin{proof}
Since $P$ is a PP over $\F_q$, Corollary~\ref{cor:PP} entails that $M$ restricts to a bijection from $U_L$ to $M(U_L)$, hence $M_0$ is well defined. In order to prove that $f_0$ is well defined, we just need to check that the $p^{k}$ elements $\{L_0(P(\zeta_i))\}_{0\le i\le p^{k}-1}$ are all distinct. If $L_0(P(\zeta_i))=L_0(P(\zeta_j))$, then 
$P(\zeta_i)=P(\zeta_j)+M(u)$ for some $u\in U_L$. But if $z=\zeta_j+u$, Eq.~\eqref{eq:coset} entails that $P(z)=P(\zeta_j)+M(u)$, implying that $\zeta_i=\zeta_j+u$ since $P$ permutes $\F_q$. However, $\zeta_i$ and $\zeta_j$ are coset representatives for the quotient $\F_q/U_L$ and so we have that $i=j$. 

It remains to verify that $P_0(P(y))=y$ for every $y\in \F_q$. Since $M$ is $p$-linearized, $M(U_L)$ is an $\F_p$-vector space of dimension $k$, hence $L_0$ is a separable $p$-linearized polynomial divding $x^q-x$. We apply Eq.~\eqref{eq:coset} twice: if $y=\zeta_i+u$ with $u\in U_L$, we have that $P(y)=P(\zeta_i)+M(u)$ and, since $M(u)\in M(U_L)$, we obtain that
$$P_0(P(y))=P_0(P(\zeta_j)+M(u))=P_0(P(\zeta_j))+M_0(M(u))=\zeta_j+u=y.$$
Therefore, $P_0$ is the inverse PP of $P$ over $\F_q$. Moreover, we have proved the following: if the permutation $P$ is an $M$-affine mapping of index $k$ with subspace $U_L$, then $P_0$ is an $M_0$-affine mapping of index $k$ with subspace $M(U_L)$. In particular, if the additive indices of $P$ and $P_0$ are $s$ and $t$, respectively, then $s\ge t$. Since $P$ is the inverse of $P_0$, we also obtain $t\ge s$, hence $s=t$. 
\end{proof}


\subsection{Involutions}
Recall that a polynomial $P\in \F_q[x]$ is an involution over $\F_q$ if $P(P(y))=y$ for every $y\in \F_q$. As a direct application of Theorem~\ref{lem:inversePP}, we obtain the following characterization of involutions over $\F_q$ induced by polynomials of the form $f(L(x))+M(x)$.

\begin{corollary}
Let $L$ be a monic $p$-linearized polynomial of degree $p^{n-k}$ that divides $x^q-x$. Let $U_L$ be the set of roots of $L$ and let $\{\zeta_i\}_{0\le i\le p^{k}-1}$ be any complete set of representatives for the quotient $\F_q/U_L$. If $P(x)=f(L(x))+M(x)\in \F_q[x]$ with $M$ a $p$-linearized polynomial, then $P$ induces an involution over $\F_q$ if and only if the following hold:
\begin{enumerate}[(i)]
\item $M(U_L)=U_L$ and $M(M(u))=u$ for every $u\in U_L$;
\item $P(P(\zeta_i))=\zeta_i$ for every $0\le i\le p^{k}-1$.
\end{enumerate} 
\end{corollary}

Recall that an element $z\in \F_q$ is a fixed point of $f\in \F_q$ if $f(z)=z$. As a direct consequence of Theorem~\ref{cor:conj}, we obtain the following family of involutions over binary fields with no fixed points.

\begin{corollary}\label{lem:inv}
Suppose that $q$ is even. Let $L$ be a monic $p$-linearized polynomial that divides $x^q-x$ and let $f\in \F_q[x]$ be a polynomial such that $L(f(L(y)))=0$ for every $y\in \F_q$. Then $P(x)=f(L(x))+x$ is an involution of $\F_q$. Moreover, if $f$ has no roots in the set $L(\F_q)$, then such involution has no fixed points.
\end{corollary}


\subsection{Linear translators}

Let $n=rk$ and  $q=p^n$.   Let $\gamma \in \fq^*, b \in \mathbb{F}_{p^k}$ and  $g: \mathbb{F}_{p^n} \rightarrow \mathbb{F}_{p^k}$.  In \cite{kyureghyan},  $g$ is  called a {\em $b$-linear translator}  if 
\[
g(x+\gamma u) = g(x) + bu \mbox{ for  all  } x \in \mathbb{F}_{q}  \  \ and \ all \ u \in \mathbb{F}_{p^k}. 
\]
More generally, $g$ is called an {\em $(i, b)$-Frobenius translator} in \cite{CPM:19} if 
\[
g(x+\gamma u) = g(x) + bu^{p^i} \mbox{ for  all  } x \in \mathbb{F}_{q}  \  \ and \ all \ u \in \mathbb{F}_{p^k}. 
\]

We remark these translators satisfy the property (\ref{eq:coset}).  Indeed,   if  we consider the subspace $U = \gamma \mathbb{F}_{p^k} \subseteq \mathbb{F}_q$  then  
\[
g(x + \gamma w)  = g(x) + M(\gamma w), \mbox{ for  all  } x \in \mathbb{F}_{q}  \  \ and \ all \ u \in \mathbb{F}_{p^k},
\]
where  $M(u) = \gamma^{-1} b u$ for $b$-linear translator and $M(u) = \gamma^{-p^i} b u^{p^i}$ for $(i, b)$-Frobenious translator.  For this reason, we introduce the following definition. 

\begin{definition}
Let $q=p^n$, let $g$ be a function from $\mathbb{F}_q$ to a subspace $U$ of $\fq$,  and  let $M \in \mathbb{F}_q[x]$ be a $p$-linearized polynomial.  Then $g$ is  an $(M, U)$-linear translator if 
\[
g(x+  u) = g(x) + M(u) \mbox{ for  all  } x \in \mathbb{F}_{q}  \  \ and \ all \ u \in U. 
\]
\end{definition}

Using the AGW criterion (see the diagram below), we can easily obtain the following result which extends the corresponding results in \cite{CPM:19, kyureghyan}.  We recall that a permutation polynomial $P(x)$ is a complete mapping if $P(x) +x$ is also a PP. 

\begin{theorem}\label{lineartranslator}
Let $U$ be a subspace of $\fq$, $g$  be a function from $\fq$ onto $U$,  $h \in \fq[x]$ such that $h(U) \subseteq U$. Assume $g$ is an $(M, U)$-linear translator.  Then $x+ h(g(x)) \in \fq[x]$ permutes $\fq$ if and only if $u + M(h(u))$ permutes $U$.  Moreover,  these PPs are complete mappings. 
\end{theorem}

\[
\xymatrixcolsep{12pc}
\xymatrix{
\fq \ar[r]^{x+ h(g(x)) }\ar[d]^{g} &  \fq \ar[d]^{g}\\
U \ar[r]^{u + M(h(u))} & U }
\]

\section{Conclusion}

In this paper we introduce the additive analogue of the index of a polynomial over finite fields. We demonstrate the usage of this parameter in the study of several classical problems on  polynomials over finite fields. We expect this will motivate the further study of other problems of polynomials in terms of their additive indices.

\end{document}